\documentclass[reqno,a4paper,11pt]{amsart}

\usepackage[all,poly]{xy}
\usepackage{amsfonts,stmaryrd}
\usepackage[mathcal]{eucal}
\usepackage{amssymb}
\usepackage{amsmath}
\usepackage{mathrsfs}
\usepackage{color}
\usepackage[pagebackref,colorlinks]{hyperref}
\usepackage{enumerate}

\theoremstyle{plain}
\newtheorem {Lem}{Lemma}
\newtheorem {The}[Lem]{Theorem}
\newtheorem {Cor}[Lem]{Corollary}
\newtheorem {Prop}[Lem]{Proposition}

\theoremstyle{remark}

\theoremstyle{definition}

\newtheorem {defn}[Lem]{Definition}

\newcommand\Label[1]{\label{#1}}
\newcommand{\ep}{\varepsilon}

\newcommand{\gm}{\mathfrak m}

\newcommand{\GL}{\operatorname{GL}}

\newif\ifcomm
\let\ifcomm\iffalse
\newcommand{\LF}{\lfloor}
\newcommand{\RF}{\rfloor}

\title{Multiple Commutator Formulas}

\author{R. Hazrat}
\address{Department of  Pure Mathematics, Queen's University Belfast, Belfast BT7 1NN, Northern Ireland, United Kingdom}
\email{r.hazrat@qub.ac.uk}

\author{Z. Zhang}
\address{Department of  Mathematics, Beijing Institute of Technology, Beijing, China}
\email{zuhong@gmail.com}

\begin{thanks}
{The first author acknowledges the support of EPSRC (Grant EP/I007784/1). The second author acknowledges the support of NSFC (Grant 10971011). The authors thank Nikolai Vavilov for suggesting the topic of the paper to them, and Anthony Bak for very useful discussions.}
\end{thanks}

\begin{document}

\begin{abstract}
Let $A$ be  a quasi-finite $R$-algebra (i.e., a direct limit of module finite algebras) with identity. Let $I_i $, $i=0,...,m$, be two-sided  ideals of $A$, $\GL_n(A,I_i)$ the principal congruence subgroup of level $I_i$ in $\GL_n(A)$ and  $E_n(A,I_i)$ be the relative elementary subgroup of level $I_i$. We prove  a multiple commutator formula 
\[
\begin{split}
\big [E_n(A,I_0),\GL_n(A,I_1),& \GL_n(A, I_2),\ldots, \GL_n(A, I_m) \big]\\
&=\big[E_n(A,I_0),E_n(A,I_1),E_n(A, I_2),\ldots, E_n(A, I_m) \big],
\end{split}
\]
which is a broad generalization of the standard commutator formulas. This result contains all the published results of commutator formulas over commutative rings and answers a problem posed by A. Stepanov and N. Vavilov (cf. Problem~4 in~\cite{NVAS}). 
\end{abstract}

\maketitle

\section*{Introduction}
Let $A$ be an associative  ring with 1, $\GL_n(A)$  the general linear group of degree $n$ over $A$, and let $E_n(A)$ be its elementary subgroup. For a two-sided ideal $I$ of $A$, we denote the principal congruence subgroup of level $I$ by $ \GL_n(A,I)$ and  the relative elementary subgroup of level $I$ by $E_n(A,I)$ (see~\S\ref{elel}). 

One of the major contributions towards non-stable $K$-theory of rings is the work of Suslin \cite{SUS,TUL} who proved that if $A$ is a {\it module finite ring\/} namely, a ring
that is finitely generated as module over its center, and $n\ge 3$
then $E_n(A)$ is a normal subgroup of $\GL_n(A)$. Thus the non-stable $K_1$, i.e., $\GL_n(A)/E_n(A)$,  can be defined.  
Later Borevich and Vavilov \cite{borvav} and Vaserstein \cite{V1}, independently,  building on Suslin's method 
established the {\it standard commutator formula}:

\begin{The}[Suslin, Borevich-Vavilov, Vaserstein]\label{standard}
Let $A$ be a module finite ring, $I$ a two-sided ideal of $A$ 
 and $n\ge 3$. Then $E_n(A,I)$ is normal in $\GL_n(A)$, i.e., 
 \[
\big[E_n(A,I),\GL_n(A)\big]=E_n(A,I).
\] 
 Furthermore
\[
\big[E_n(A),\GL_n(A,I)\big]=E_n(A,I).
\] 
\end{The}

One natural question raised here is whether one has a ``finer'' mixed commutater formulas involving two ideals. In fact this had already been established by Bass for general linear groups of degrees sufficiently larger than the stable rank when he proved his celebrated classification of subgroups of $\GL_n$ normalized by $E_n$ (see \cite[Theorem~4.2]{Bass2}). 

\begin{The}[Bass]\label{hhggff}
Let $A$ be a ring, $I,J$ two-sided ideals of $A$ and $n \geq \max(sr(R)+1,3)$. Then 
\[ \big [E_n(A,I),\GL_n(A,J) \big ]=\big [E_n(A,I),E_n(A,J) \big ]. \]
\end{The}

Later Mason and Stothers building on Bass' result prove (\cite[Theorem~3.6, Corollary~3.9]{MAS3}, and \cite[Theorem~1.3]{MAS1}):

\begin{The}[Mason-Stothers]
Let $A$ be a ring, $I,J$ two-sided ideals of $A$ and $n \geq \max(sr(R)+1,3)$. Then 
\[ \big [\GL_n(A,I),\GL_n(A,J) \big ]=\big [E_n(A,I),E_n(A,J) \big ]. \]
\end{The}

There are (counter)examples that the Mason-Strothers Theorem does not hold for general module finite rings \cite{B4}. However recently Stepanov and Vavilov \cite{NVAS,NVAS2} proved Bass' Theorem~\ref{hhggff} for any commutative ring and $n\geq 3$ and the authors using Bak's localization and patching method extend it to all module finite rings \cite{RHZZ1}. We refer to this as the {\it generalized commutator formula}. In \cite{NVAS} it is asked whether one can establish a multiple commutator formula, namely  for a commutative ring $R$, $I_i $, $i=0,...,m$, ideals of $R$ and $n\geq 3$,  whether 
\begin{equation}\label{minfr}
\begin{split}
\big [E_n(R,I_0),\GL_n(R,I_1),& \GL_n(R, I_2),\ldots, \GL_n(R, I_m) \big]\\
&=\big[E_n(R,I_0),E_n(R,I_1),E_n(R, I_2),\ldots, E_n(R, I_m) \big],
\end{split}
\end{equation}
is valid which is a broad generalization of the standard/generalized commutator formulas. Here for simplicity we write 
$[A_1,A_2,A_3,\dots,A_n]$ for $\big[\dots\big[[A_1,A_2],A_3\big],\dots,A_n\big]$ (see \S\ref{sub:1.4}). Questions of this type arise from the study of subnormal subgroups of $\GL_n$ from one hand and the nilpotent structure of nonstable $K_1$ from the other hand 
(see~\cite[\S10 and \S12]{RN} for a survey on these topics).

In this paper we prove Formula~(\ref{minfr}) for quasi finite rings (which include module finite and commutative rings) (see Corollary~\ref{comain}). In particular this result contains all the published results of commutator formulas over commutative rings.  In fact in Theorem~\ref{main2} we show that the multiple commutator formulas are valid for any meaningful way of the distribution of commutators.

To establish these results, we use the general ``yoga of commutators'' which are developed in~\cite{RHZZ1} and~\cite{RVZ1} based on the work of Bak on the localization and patching in general linear groups (see \cite{B4,yoga} and \cite[\S13]{RN}).  In order to utilize this method, one needs to overcome two problems. First to devise an appropriate conjugation calculus to approach the identity~(\ref{minfr}) and then perform the actual calculations. Both of these are equally challenging as the nature of conjugation calculus depends on the problem in hand. In fact the term  yoga of commutators is chosen
to stress the overwhelming feeling of technical strain and exertion.
 However once this is done for general linear groups, one can  
adapt the approach to more complex settings, such as general quadratic groups and  
Chevalley groups. These shall be established in a sequel to this paper.


\section{Preliminaries}
In this section we fix some notations. At the same time, we list some preliminary results concerning the localization and patching method without proofs. We refer to Bak's original paper \cite{B4} or a survey version in \cite[\S13]{RN} for details.

\subsection{}\Label{sub:1.1}
Let $R$ be a commutative ring with $1$,  $S$ a multiplicative closed system in $R$ and $A$ an $R$-algebra.  Then $S^{-1}R$ and $S^{-1}A$ denote the corresponding localization. In the current paper, we mostly use localization with respect to the following two types of multiplicative systems. 

\noindent 1.) For any $s\in R$, the multiplicative system generated by $s$ is defined as 
$$\langle s \rangle = \{1, s ,s^2,\ldots \}.$$
The localization with respect to multiplicative system $\langle s\rangle$ is usually denoted by $R_s$ and $A_s$. Note that, for any $\alpha \in R_s$, there exists an integer $n$ and an element  $a\in R$ such that 
$\alpha=a/{s^n}$.

\noindent 2.) If $\gm$ is a maximal ideal of $R$, and $S=R\backslash \gm$ a multiplicative system, then we denote the localization with respect to $S$ by $R_\gm$ and $A_\gm$. 

For a multiplicative system $S$, the canonical localization map with respect to $S$ is denoted by $\theta_S: R\to S^{-1}R$. For the special cases mentioned above, we write 
$\theta_s: R\to R_s$ and $\theta_M: R\to R_M$, respectively.

\subsection{}\Label{sub:1.2} An $R$-algebra $A$ is called {\em module finite} over $R$, if $A$ is finitely generated as an $R$-module. An $R$-algebra $A$ is  called {\em quasi-finite} over $R$ if there is a direct system of module finite  $R$-subalgebras $A_i$ of $A$ such that $\varinjlim A_i=A$. 
\begin{Prop}\Label{Prop:01}
An $R$-algebra $A$ is quasi-finite over R if and only if it satisfies the following equivalent conditions:
\begin{itemize}
\item[(1)] There is a direct system of subalgebras $A_i/R_i$ of $A$ such that each $A_i$ is module finite over $R_i$ and such that 
$\varinjlim R_i=R$ and $\varinjlim A_i=A$.
\item[(2)] There is a direct system of subalgebras $A_i/R_i$ of $A$ such that each $A_i$ is module finite over $R_i$ and each $R_i$ is finitely generated as a $\mathbb Z$-algebra and such that $\varinjlim R_i=R$ and $\varinjlim A_i=A$.
\end{itemize}
\end{Prop}

\subsection{}\Label{sub:1.4} Let $G$ be a group. For any $x,y\in G$,  $^xy=xyx^{-1} $  denotes the left $x$-conjugate of $y$. Let  $[x,y]=xyx^{-1}y^{-1}$ denote the commutator of $x$ and $y$.  Sometimes the double commutator $[[x, y], z]$ will be denoted simply by $[x, y, z]$ and $$\big[[A,B],C\big]=[A,B,C].$$ Thus we write $[A_1,A_2,A_3,\dots,A_n]$ for $\big[\dots\big[[A_1,A_2],A_3\big],\dots,A_n\big]$ and call it the {\it standard form} of the multiple commutator formulas.

The following formulas will be used frequently (sometimes without giving a reference to them),
\begin{itemize}
\item[(C1)] $[x,yz]=[x,y]({}^y[x,z])$;
\smallskip
\smallskip
\item[(C$1^+$)] 
An easy induction, using identity (C1), shows that 
$$\big [x,\prod_{i=1}^k u_i]=\prod_{i=1}^k {}^{\prod_{j=1}^{i-1}u_j}[x,u_{i}],$$ where by convention $\prod_{j=1}^0 u_j=1$.
\item[(C2)] $[xy,z]=(^x[y,z])[x,z]$;
\smallskip
\item[(C$2^+$)] 
As in (C$1^+$), we have  
$$\big [\prod_{i=1}^k u_i,x\big]=\prod_{i=1}^k {}^{\prod_{j=1}^{k-i}u_j}[u_{k-i+1},x].$$ 
\smallskip
\item[(C3)] (the Hall-Witt identity): 
 ${}^{x}\big[[x^{-1},y],z\big]\, \, {}^{z}\big[[z^{-1},x],y\big]\, \, {}^{y}\big[[y^{-1},z],x\big]=1$;
\smallskip
\item[(C4)] $[x,^yz]=^y[^{y^{-1}}x,z]$;
\smallskip
\item[(C5)] $[^yx,z]=^{y}[x,^{y^{-1}}z]$.
\smallskip
\item[(C6)] If $H$ and $K$ are subgroups of $G$, then $[H,K]=[K,H]$. 
\end{itemize}

\subsection{} \label{elel} For any  associative  ring $A$, $\GL_n(A)$ denotes the general linear group of $A$,  and $E_n(A)$ denotes the elementary subgroup of $\GL_n(A)$. Let $I$ be any two-sided ideal of $A$. If $\rho_I $ denotes the natural ring homomorphism $A\to A/I$, then $\rho_I$ induces a  group homomorphism, denoted also by $\rho_I$, $\rho_I:\GL_n(A)\to \GL_n(A/I)$.   The congruence subgroup of level $I$ is defined as $\GL_n(A,I)=\ker(\rho_I: \GL_n(A)\to \GL_n(A/I))$. The {\it elementary subgroup of level} $I$ is, by definition, the subgroup generated by all elementary matrices $e_{i,j}(\alpha)$ with $\alpha\in I$. The normal closure of $E_n(I)$ in $E_n(A)$, the {\it relative elementary subgroup of level} $I$, is denoted by $E_n(A, I)$. We use $E_n^L(I)$ to denote the subset of $E_n(I)$, which can be represented as the product $L$ elementary matrices. $E_n^L(I)$ is not necessarily a group.

We have the following relations among elementary matrices which will be used in the paper:
\begin{itemize}
\item[(E1)] $e_{i,j}(a)e_{i,j}(b)=e_{i,j}(a+b).$
\item[(E2)] $[e_{i,j}(a),e_{k,l}(b)]=1$ if $i\not = l, j \not = k$.
\item[(E3)] $[e_{i,j}(a),e_{j,k}(b)]=e_{i,k}(ab)$ if $i \not = k$.
\end{itemize}

\subsection{}\Label{sub:1.3}
$\GL_n$ and $E_{n}$ define two functors from the category of associative rings to the category of groups. These functors commute with direct limits. In another words, let $A_i$ be an inductive system of rings, and $A=\varinjlim A_i$. Then 
$$\GL_n(A)=\GL_n(\varinjlim A_i)\cong\varinjlim \GL_n(A_i)\quad \text{and} \quad E_n(\varinjlim A_i)\cong\varinjlim E_n(A_i).$$ Also, if $J$ is an ideal of $A$, then there are ideals $J_i$ of $A_i$ such that $J=\varinjlim J_i$ and 
$$\GL_n(A,J) =  \GL_n(\varinjlim A_i,\varinjlim J_i)\cong\varinjlim \GL_n(A_i,J_i).$$
By Proposition~\ref{Prop:01} and the above observation, we may reduce some of our problems to the case of  the Noetherian rings.
Let $S$ be a multiplicative system in $R$, $R_s$ with $s\in S$ is an inductive system with respect to the localization map : $\theta_t: R_s\to R_{st}$. If $\mathcal F$ is a functor commuting with direct limits (here $\GL_n$ and $E_n$), then 
$$
\mathcal F(S^{-1}R)=\varinjlim F(R_s).
$$
This allows us to reduce our problems  in any localization to the localization in one element. Starting from Section~\ref{ghgh}, we will be working in the ring $A_t$. However, eventually we need to return to the ring $A$. The following Lemma provides a way to ``pull back'' elements from $\GL_n(A_t)$ to $\GL_n(A)$. 

\begin{Lem}{\cite[Lemma~4.10]{B4}}\Label{Lem:03}
Let $A$ be a module finite $R$-algebra, where $R$  is a commutative Noetherian ring. Then for any $t\in R$,  there exists a positive integer $l$ such that the homomorphism 
$\theta_t: \GL_n(A, t^l A)\longrightarrow \GL_n(A_t)$ is injective.
\end{Lem}

\begin{defn}\label{nndef}
Let $A$ be an $R$-algebra, $I$ a two-sided ideal of $A$, $t\in R$, and $l$ a positive integer. Define 
$E_n(t^lA,t^lI)$ to be  a subgroup of $E_n(A,t^lI)$ generated by 
$$
{}^ee_{i,j}(t^l\alpha) \quad \text{for all}\quad \alpha\in I, e\in E_n(t^lA) \text{ and } 1\leq i\not = j\le n.
$$
Here by $t^lI$, we are considering the image of $t \in R$ in $A$ under the algebra structure homomorphism. It is clear that $t^lI$ is also an ideal of $A$. 

For any element $\alpha\in A$, we use $E_n(t^lA, t^l\alpha)$ to denote the subgroup generated by $$
{}^{e}e_{i,j}(t^l\alpha) \quad \text{for all}\quad  e\in E_n(t^lA) \text{ and } 1\le i,j\le n.
$$
\end{defn}

From the definition, it is clear that $E_n(t^lA,t^lI)$ is normalized by $E_n(t^lA)$.  This will be used throughout out the calculations. Also, by Lemma~\ref{Lem:03},  both $E_n(t^lA, t^lI)$ and $E_n(t^lA, t^l\alpha)$ are embedded in $\GL_n(A_t)$ for  a sufficiently large integer $l$. This fact will be used in Theorem~\ref{main}.

\subsection{}\label{fghdjs}

Finally we need the following elementary conjugation calculus, Lemmas~7, 8 and 11  from~\cite{RHZZ1}, respectively. 
Note that in Equations~\ref{lem5}, \ref{lem8} and~\ref{lem11} the calculations take place in the group $E_n(A_t)$.

\begin{Lem}[cf. ~\cite{RHZZ1}]\Label{LemHZ}
Let $A$ be a module finite $R$-algebra, $I,J$ two-sided ideals of $A$, $a,b,c\in A$ and  $t\in R$. If $m,l$ are given, there is an integer $p$ such that 
\begin{equation}\label{lem5}
  {}^{ E_n^1(\frac{c}{t^m})}E_n( t^pA,t^{ p}\langle  a\rangle)\subseteq E_n( t^lA,t^{ l}\langle  a\rangle),
\end{equation}
there is an integer $p$ such that 
\begin{equation}\label{lem8}
  {}^{ E_n^1(\frac{c}{t^m})}\big[E_n( t^pA,t^{ p}\langle  a\rangle), E_n(t^pA,t^{ p}\langle  b\rangle)\big]\subseteq \big[E_n( t^lA,t^{ l}\langle  a\rangle), E_n(t^lA,t^{ l}\langle  b\rangle)\big],
\end{equation}
and there is an integer $p$ such that 
\begin{equation}\label{lem11}
\Big[E_n(t^p A, t^{ p}  I), E^1_n\big(\frac{J}{t^m}\big)\Big]\subseteq \big[E_n(t^l A, t^{ l}  I), E_n(t^lA,t^{ l} J)\big].
\end{equation}
\end{Lem}

By   Lemma~\ref{LemHZ}, one obtains the following result easily. The proof is left to the reader.
\begin{Lem}\Label{LemHZ1}
Let $A$ be a module finite $R$-algebra, $I,J$ two-sided ideals of $A$, $a,b,c\in A$ and  $t\in R$. If $m,l, L$ are given, there is an integer $p$ such that 
\begin{equation}\label{lemnew}
\Big[E_n(t^p A, t^{ p}  I), {}^{E_n^L\big(\frac{A}{t^m}\big)}E^1_n\big(\frac{J}{t^m}\big)\Big]\subseteq \big[E_n(t^l A, t^{ l}  I), E_n(t^lA,t^{ l} J)\big].
\end{equation}
\end{Lem}

\section{Commutator subgroups}\label{ghgh}
In this section we study the relations between multiple commutator subgroups over  a quasi-finite algebra. The proofs are heavily depend on the computation in \cite{RHZZ1} (see Lemma~\ref{LemHZ}).
Throughout the section ideals are two sided and we assume $n\geq  3$ for any general linear group $\GL_n$. 

We record the following well-known lemma originally established by Suslin and Vaserstein (cf. \cite[Lemma 4.8]{B4}) which is needed in computations. 
\begin{Lem}\label{Engenerator}
Let $A$ be a ring and $I$ a two-ideal of $A$. Then $E_n(A,I)$ is generated as a group by the elements $${}^{e_{i,j}(a)}e_{j,i}(\alpha),$$
where $i\ne j$, $a\in A$ and $\alpha\in I$.
\end{Lem}
Using Lemma~\ref{Engenerator} it is not hard to prove that $E_n(A, I^2)\subseteq E_n(I)$ (see~\cite[Corollary~4.9]{B4}  and~\cite[Proposition~2]{TITS}).
This containment can be slightly generalized to the case of two ideals. The following Lemma will be used throughout our calculations. 
\begin{Lem}\label{Lem:Habdank}
Let $A$ be a ring and $I, J$ be two-ideals of $A$. Then 
$$
E_n{(A, IJ+JI)}\subseteq \big[E_n(I),E_n(J) \big] \subseteq \big[E_n(A,I),E_n(A,J)\big] \subseteq \GL_n(A,IJ+JI).
$$
\end{Lem}
\begin{proof}
The proof is routine by using Lemma~\ref{Engenerator} and is left to the reader.
\end{proof}
\begin{Lem}\Label{Comgenerator}
Let $A$ be a ring and $I, J$ be two-ideals of $A$. Then $\big[E_n(A,I), E_n(A,J)\big]$ is generated as a group by the elements of the form 
\begin{equation}\label{yyttrree}
^{c}\big[e_{j,i}(\alpha),{}^{e_{i,j}(a)}e_{j,i}(\beta)\big],\quad {}^c\big[e_{j,i}(\alpha),e_{i,j}(\beta)\big], \quad   {}^c e_{i,j}(\alpha\beta), \quad \text{ and} \quad {}^c e_{i,j}(\beta\alpha),
\end{equation}
where $1\leq i\ne j\leq n$, $\alpha\in I$, $\beta\in J$, $a \in A$ and $c\in E_n(A)$.
\end{Lem}
\begin{proof}
A typical generator of $\big[E_n(A,I), E_n(A,J)\big]$ is of the form 
$[e,f]$, where $e \in E_n(A,I)$ and $f \in E_n(A,J)$. Thanks to Lemma~\ref{Engenerator}, we may assume that $e$ and $f$ are products of elements the form 
$$
e_{i}=^{e_{p',q'}(a)} e_{q',p'}(\alpha)  \quad \text{and} \quad 
f_{j}=^{e_{p,q}(b)} e_{q,p}(\beta),
$$
where $a,b\in A$, $\alpha\in I $ and $\beta\in J$, respectively. Applying (C$1^+$) and then (C$2^+$), one gets that
$\big[E_n(A,I), E_n(A,J)\big]$ is generated by the elements of the form 
$$^{c}\big[^{e_{i',j'}(a)}e_{j',i'}(\alpha),{}^{e_{i,j}(b)}e_{j,i}(\beta)\big],$$
where $c\in E_n(A)$. Furthermore, 
$$
^c\big[^{e_{i',j'}(a)}e_{j',i'}(\alpha),{}^{e_{i,j}(b)}e_{j,i}(\beta)\big]=^{ce_{i',j'}(a)}\big[e_{j',i'}(\alpha),{}^{e_{i',j'}(-a)e_{i,j}(b)}e_{j,i}(\beta)\big].
$$
The normality of $E_n(A,J)$ implies that ${}^{e_{i',j'}(-a)e_{i,j}(b)}e_{j,i}(\beta)\in E_n(A,J)$, which is a product of  $^{e_{p,q}(a)}e_{q,p}(\beta)$, $a\in A$ and $\beta \in J$ by Lemma~\ref{Engenerator}. Again by (C$1^+$), one reduces the proof to the case 
of showing that
$$
\big[e_{i',j'}(\alpha),{}^{e_{i,j}(a)}e_{j,i}(\beta)\big]
$$
is a product of the generators listed in~(\ref{yyttrree}). We need to consider following cases:
\begin{itemize}
\item If $i'=j,j'=i$: Then there is nothing to proof.

\smallskip 

\item if $i'=j, j'\ne i$:
\begin{eqnarray*}
\big[e_{j,j'}(\alpha),{}^{e_{i,j}(a)}e_{j,i}(\beta)\big]&=&{}^{e_{i,j}(a)}\big[{}^{e_{i,j}(-a)}e_{j,j'}(\alpha ), e_{j,i}(\beta)\big]\\
&=&{}^{e_{i,j}(a)}\big[[{e_{i,j}(-a)},e_{j,j'}(\alpha )]e_{j,j'}(\alpha ), e_{j,i}(\beta)\big]\\
&=&{}^{e_{i,j}(a)}\big[e_{i,j'}(-a\alpha )e_{j,j'}(\alpha ), e_{j,i}(\beta )\big].
\end{eqnarray*}
Applying now (C2),
\begin{eqnarray*}
[e_{i,j'}(-a\alpha )e_{j,j'}(\alpha ), e_{j,i}(\beta )]&=&
\big({}^{e_{i,j'}(-a\alpha )}[e_{j,j'}(\alpha ), e_{j,i}(\beta )]\big)[e_{i,j'}(-a\alpha ), e_{j,i}(\beta )]\\
&=&[e_{i,j'}(-a\alpha ), e_{j,i}(\beta )]\\
&=&[ e_{j,i}(\beta ),e_{i,j'}(-a\alpha )]^{-1}\\
&=&e_{j,j'}(-\beta a \alpha)^{-1}\\
&=&e_{j,j'}(\beta a \alpha)
\end{eqnarray*}
Thus 
$$
\big[e_{j,j'}(\alpha),{}^{e_{i,j}(a)}e_{j,i}(\beta)\big]={}^{e_{i,j}(a)}e_{j,j'}(\beta a \alpha)
$$
which satisfies the lemma.
\item if $i'\ne j, j'= i$: The argument is similar to the previous case.

\smallskip 

\item if $i'\not =j, j' \not = i$: We consider four cases: 
\begin{itemize} 

\smallskip 

\item if $i'=i,j'=j$: 
$$
\big[e_{i,j}(\alpha),{}^{e_{i,j}(a)}e_{j,i}(\beta)\big]={}^{e_{i,j}(a)}\big[e_{i,j}(\alpha),e_{j,i}(\beta)\big].
$$
\item if $i'=i,j'\ne j$:
\begin{eqnarray*}
\big[e_{i,j'}(\alpha),{}^{e_{i,j}(a)}e_{j,i}(\beta)\big]&=&{}^{e_{i,j}(a)}\big[e_{i,j'}(\alpha),e_{j,i}(\beta)\big]\\
&=&{}^{e_{i,j}(a)}e_{j,j'}(-\beta\alpha).
\end{eqnarray*}

\item if $i'\ne i,j'=j$:
 \begin{eqnarray*}
\big[e_{i',j}(\alpha),{}^{e_{i,j}(a)}e_{j,i}(\beta)\big]&=&{}^{e_{i,j}(a)}\big[e_{i',j}(\alpha),e_{j,i}(\beta)\big]\\
&=&{}^{e_{i,j}(a)}e_{i,i'}(\alpha\beta).
\end{eqnarray*}

\item if $i'\ne i, j'\ne j$: 
$$
\big[e_{i',j'}(\alpha),{}^{e_{i,j}(a)}e_{j,i}(\beta)\big]=1.
$$

\end{itemize}
\end{itemize}
This finishes the proof.
\end{proof}

\ifcomm{
\begin{proof}
Let 
$$
 \Bigl[e_{i',j'}(t^p\alpha), {}^{e_{i,j} (\frac{a}{t^m})}e_{i'',j''}(\frac{\beta}{t^m})\Bigr]\in[E^1_n(t^p I), ^{E_n^1(\frac{A}{t^m})}E^1_n(\frac{J}{t^m})]
$$
for  $\alpha\in I$, $\beta\in J$, $a\in A$ and some positive integers $i,i',i'',j,j',j''$. Then
\begin{eqnarray*}
&& \Bigl[e_{i',j'}(t^p\alpha), {}^{e_{i,j} (\frac{a}{t^m})}e_{i'',j''}(\frac{\beta}{t^m})\Bigr]\\
&=&\Bigl[e_{i',j'}(t^p\alpha), e_{i'',j''}(\frac{\beta}{t^m})[ e_{i'',j''}(-\frac{\beta}{t^m}),{e_{i,j} (\frac{a}{t^m})}]\Bigr]\\
&&\text{By identity (C1) }\\
&=&[e_{i',j'}(t^p\alpha), e_{i'',j''}(\frac{\beta}{t^m})]\Bigl( {}^{e_{i'',j''}(\frac{\beta}{t^m})}\Bigl[e_{i',j'}(t^p\alpha),[ e_{i'',j''}(-\frac{\beta}{t^m}),{e_{i,j} (\frac{a}{t^m})}]\Bigr]\Bigr)
\end{eqnarray*}
By Lemma~\ref{Lem:06}, we may choose a large enough $p$ such that 
$$[e_{i',j'}(t^p\alpha),e_{i'',j''}(\frac{\beta}{t^m})]\in[E_n(t^l A, t^{ l}  I), E_n(t^lA,t^{ l} J)].$$
We claim that 
$${}^{e_{i'',j''}(\frac{\beta}{t^m})}\Bigl[e_{i',j'}(t^p\alpha),[ e_{i'',j''}(-\frac{\beta}{t^m}),{e_{i,j} (\frac{a}{t^m})}]\Bigr]\in[E_n(t^l A, t^{ l}  I), E_n(t^lA,t^{ l} J)].$$
We have
\begin{eqnarray*}
&&{}^{e_{i'',j''}(\frac{\beta}{t^m})}\Bigl[e_{i',j'}(t^p\alpha),[ e_{i'',j''}(-\frac{\beta}{t^m}),{e_{i,j} (\frac{a}{t^m})}]\Bigr]\\
&=& {}^{e_{i'',j''}(\frac{\beta}{t^m}){e_{i,j} (-\frac{a}{t^m})}{e_{i,j} (\frac{a}{t^m})}}\Bigl[e_{i',j'}(t^p\alpha),[ e_{i'',j''}(-\frac{\beta}{t^m}),{e_{i,j} (\frac{a}{t^m})}]\Bigr]\\
&=& {}^{e_{i'',j''}(\frac{\beta}{t^m}){e_{i,j} (-\frac{a}{t^m})}}\underbrace{ \Bigl({}^{e_{i,j} (\frac{a}{t^m})}\Bigl[e_{i',j'}(t^p\alpha),[ e_{i'',j''}(-\frac{\beta}{t^m}),{e_{i,j} (\frac{a}{t^m})}]\Bigr]\Bigr)}_{e}.\\
\end{eqnarray*}
Hall-Witt identity implies
\begin{eqnarray*}
e&=&^{e_{i',j'}(t^p\alpha)}\Bigl[[{e_{i,j} (\frac{a}{t^m})},{e_{i',j'}(t^p\alpha)}],{e_{i'',j''}(-\frac{\beta}{t^m})}\Bigr] \times\\
&&\times {}^{e_{i'',j''}(-\frac{\beta}{t^m})}\Bigl[[{e_{i',j'}(t^p\alpha)},{e_{i'',j''}(-\frac{\beta}{t^m})}],{e_{i,j} (\frac{a}{t^m})}\Bigr].
\end{eqnarray*}
By Lemma~\ref{Lem:06}, for any given $p'$, there is a sufficiently large choice of $p'$ such that that  $[{e_{i,j} (\frac{a}{t^m})},{e_{i',j'}(t^p\alpha)}]$  belongs to $[E_n(t^{p'} A, t^{ p'}  I), E_n(t^{p'}A,t^{ p'} A)]=E_n(t^{p'} A, t^{ p'}  I)$. Hence 
$$
\Bigl[[{e_{i,j} (\frac{a}{t^m})},{e_{i',j'}(t^p\alpha)}],{e_{i'',j''}(-\frac{\beta}{t^m})}\Bigr]\in[E_n(t^{p'} A, t^{ p'}  I), E^1_n(\frac{J}{t^m})].
$$
By Lemma~\ref{Lem:left1}, we may choose a sufficiently large $p'$ such that
$$[E_n(t^{p'} A, t^{ p'}  I), E^1_n(\frac{J}{t^m})] \subseteq [E_n(t^l A, t^{ l}  I),E_n(t^lA,t^{ l} J)].$$
Therefore
$$\Bigl[[{e_{i,j} (\frac{a}{t^m})},{e_{i',j'}(t^p\alpha)}],{e_{i'',j''}(-\frac{\beta}{t^m})}\Bigr] \in [E_n(t^l A, t^{ l}  I),E_n(t^lA,t^{ l} J)].$$
It shows that the first factor of $e$ 
$$^{e_{i',j'}(t^p\alpha)}\Bigl[[{e_{i,j} (\frac{a}{t^m})},{e_{i',j'}(t^p\alpha)}],{e_{i'',j''}(-\frac{\beta}{t^m})}\Bigr] \in
[E_n(t^l A, t^{ l}  I),E_n(t^lA,t^{ l} J)].$$

Now we consider the second factor of $e$. Notice that given any $p'$, there is a large enough $p$ such that 
$$[{e_{i',j'}(t^p\alpha)},{e_{i'',j''}(-\frac{\beta}{t^m})}] \in [E_n(t^{p'} A, t^{ p'}  I),E_n(t^{p'}A,t^{ p'} J)] .$$
Then Lemma~\ref{Lemm:04'} implies that there is a sufficiently large $p$ such that
$${}^{e_{i'',j''}(-\frac{\beta}{t^m})}\Bigl[[{e_{i',j'}(t^p\alpha)},{e_{i'',j''}(-\frac{\beta}{t^m})}],{e_{i,j} (\frac{a}{t^m})}\Bigr]
\in [E_n(t^l A, t^{ l}  I),E_n(t^lA,t^{ l} J)].$$
We proves that for any given $l$, there is a large enough $p$ such that 
$$e\in [E_n(t^l A, t^{ l}  I),E_n(t^lA,t^{ l} J)].$$
Hence 
$${}^{e_{i'',j''}(\frac{\beta}{t^m}){e_{i,j} (-\frac{a}{t^m})}}e\in [{}^{E_n^1(\frac{A}{t^m})}{}^{E_n^1(\frac{A}{t^m})}E_n(t^l A, t^{ l}  I),{}^{E_n^1(\frac{A}{t^m})}{}^{E_n^1(\frac{A}{t^m})}E_n(t^lA,t^{ l} J)]$$
By using Lemma~\ref{Lemm:04'} twice, one proves the claim.
 \end{proof}
 }
 \else{}\fi

Denote by $E^L_n\big(\frac{A}{t^m}, \frac{K}{t^m}\big)$ the product of $L$ elements (or fewer) of the form $^{E_n^1(\frac{A}{t^m})}E^1_n\big(\frac{K}{t^m}\big)$ (see also \S\ref{elel}). In the following two Lemmas, as in Lemma~\ref{LemHZ}, all the calculations take place in the fraction ring $A_t$ (see \S\ref{fghdjs}). All the subgroups used in the Lemmas, such as $E_n(A,I)$ or $\GL_n(A,J)$ are in fact the images of these groups in $\GL_n(A_t)$ under the ring homomorphisms $A\rightarrow A_t$. This allows us to use Lemmas such as Lemma~\ref{Lem:Habdank} and the generalized commutator formula on these subgroups, precisely because these are homomorphic images of the similar subgroups in $\GL_n(A)$ which Lemma~\ref{Lem:Habdank}, etc. hold. 

\begin{Lem}\Label{Lem:New1}
Let $A$ be a module finite $R$-algebra, $I,J$ two-sided ideals of $A$, and  $t\in R$.  If $e\in \GL_n(A_t,J_t)$, there is an integer $p$ such that for any $g\in \GL_n(A,t^p I )$ 
$$
[e,g]\in \GL_n\big({A, t^l (IJ+JI)}\big).
$$
\end{Lem}
\begin{proof}
Note that all the entries of $g-1$ and  $g^{-1}-1$ are in $t^pI$ (to emphasize our convention, they are in the image of $t^pI$ under the homomorphism $\theta:A\rightarrow A_t$) and all the entries of $e-1$ and $e^{-1}-1$ are in $J_t$. Choose $k\in \mathbb N$  such that one can write all the entries of $e-1$ and $e^{-1}-1$ of the form $j/t^k$, $j \in J$.
Let 
\begin{eqnarray*}
g=1+\ep& \quad \text{and}\quad &g^{-1}=1+\ep'\\
e=1+\delta& \quad \text{and}\quad &e^{-1}=1+\delta'.
\end{eqnarray*}
A straightforward computation shows that 
\begin{align*}
& \ep+\ep'+\ep\ep'=\ep+\ep'+\ep'\ep=0\\
& \delta +\delta'+\delta\delta'=\delta +\delta'+\delta'\delta=0.
\end{align*}
By the equalities above, one has
$$[e,g]=[1+\delta,1+\ep]=1+\delta'\ep'+\ep\delta'+\ep\delta'\ep'+\delta\delta'\ep'+\delta\ep\delta'+\delta\ep\delta'\ep'.
$$
So  the entries of $[e,g]-1$  belong to  $t^{p-2k}(IJ +JI)$. We finish the proof by choosing $p\ge l+2k$. 
\end{proof}

The following lemma is crucial for proving the main result, i.e., Theorem~\ref{main} of this paper. 

\begin{Lem}\Label{Lem:08}
Let $A$ be a module finite $R$-algebra, $I,J,K$  two-sided ideals of  $A$ and $t\in R$.  For any given $e_2\in E_n(A_t,K_t)$ and an integer $l$, there is a  sufficiently large integer $p$, such that 
\begin{equation}\label{eqn:l1}
[e_1,e_2]\in \Big[\big[E_n(A,t^lI), E_n(A,t^lJ)\big], E_n(A,t^lK)\Big].
\end{equation}
where  $e_1\in [E_n(t^pI), E_n(A,J)]$.
\end{Lem}
\begin{proof}
%
For any given $e_2\in E_n(A_t, K_t)$, one may find some positive integers $m$ and $L$, such that
$$
e_2\in E^L_n\big(\frac{A}{t^m},\frac{K}{t^m}\big).
$$
Applying the identity (C1$^+$) and repeated application of (\ref{lem5}) in Lemma~\ref{LemHZ}, we reduce the problem to show that 
$$
\Big[[E_n(t^pI), E_n(A,J)], {}^ce_{i',j'}(\frac{\gamma}{t^m})\Big] \subseteq \Big[\big[E_n(A,t^lI), E_n(A,t^lJ)\big], E_n(A,t^lK)\Big],
$$
where $c\in E_n^1(\frac{A}{t^m})$ and  $\gamma\in K$.
We further decompose   $e_{i',j'}(\frac{\gamma}{t^m})=[e_{i',k}(t^{p'}),e_{k,j'}(\frac{\gamma}{t^{m+p'}})]$ for some integer $p'$. Then
\begin{equation*}
 \Big [e_1, {}^ce_{i',j'}(\frac{\gamma}{t^m})\Big]=
\Big [e_1, \big [{}^ce_{i',k}(t^{p'}),{}^c e_{k,j'}(\frac{\gamma}{t^{m+p'}})\big]\Big].
\end{equation*}
We use a variant of the Hall-Witt identity (see (C3)) 
$$ \big[x,[y^{-1},z]\big]=  {}^{y^{-1}x}\big[[x^{-1},y],z\big]\, \, {}^{y^{-1}z}\big[[z^{-1},x],y\big],$$ to obtain 
\begin{align}
\Big [e_1,& \big [{}^ce_{i',k}(t^{p'}),{}^c e_{k,j'}(\frac{\gamma}{t^{m+p'}})\big]\Big]=\notag\\
=&{}^{y^{-1}x}\bigg[\Big[e_1^{-1}, {}^ce_{i',k}(-t^{p'})\Big],{}^c e_{k,j'}(\frac{\gamma}{t^{m+p'}})\bigg] \times \notag\\\label{eqn:99}
&{}^{y^{-1}z}\bigg[\Big[{}^c e_{k,j'}(\frac{-\gamma}{t^{m+p'}}), e_1\Big],{}^ce_{i',k}(-t^{p'})\bigg],
\end{align}
where $x=e_1$, $y={}^ce_{i',k}(-t^{p'})$, $z={}^c e_{k,j'}(\frac{\gamma}{t^{m+p'}})$ and as before $c\in E_n^1(\frac{A}{t^{m}})\subseteq E_n^1(\frac{A}{t^{m+p'}})$. 
We will look at each of the two factors of~(\ref{eqn:99}) separately.  

By (\ref{lem5}) in Lemma~\ref{LemHZ}, for any given $p''$, one may find a sufficiently large $p'$ such that 
\begin{equation}\Label{eqn:100}
y={}^ce_{i',k}(-t^{p'})\in E_n(t^{p''}A, t^{p''}A) \subseteq E_n(A).
\end{equation}
Then 
\begin{eqnarray*}
\Big[e_1^{-1}, {}^ce_{i',k}(-t^{p'})\Big] &\in&\big[[E_n(t^p I), E_n(A,J)], E_n(A)\big]\\
&\subseteq &\big [\GL_n(A, t^p(IJ+JI)), E_n(A)\big]\\
&\subseteq & E_n(A, t^p(IJ+JI)).
\end{eqnarray*}
Set $p_1=p$. Thanks to Lemma~\ref{Lem:Habdank}, 
$$E_n\big(A, t^{p_1}(IJ+JI)\big)\subseteq \Big[E_n(t^{\LF {\frac{p_1}{2}\RF}}A), E_n\big(t^{\LF {\frac{p_1}{2}\RF}}(IJ+JI)\big)\Big]\subseteq 
E_n\big(t^{\LF {\frac{p_1}{2}\RF}}A, t^{\LF {\frac{p_1}{2}\RF}}(IJ+JI)\big).
$$
Hence  we obtain that
$$
{}^{y^{-1}x}\bigg[\Big[e_1^{-1}, {}^ce_{i',k}(-t^{p'})\Big],{}^c e_{k,j'}(\frac{\gamma}{t^{m+p'}})\bigg]
\in {}^{y^{-1}x}\bigg[E_n(t^{\LF {\frac{p_1}{2}\RF}}A, t^{\LF {\frac{p_1}{2}\RF}}(IJ+JI)),{}^c e_{k,j'}(\frac{\gamma}{t^{m+p'}})\bigg], $$
where $x\in [E_n(t^{p_1} I), E_n(A,J)]$, $y\in  E_n(t^{p''}A, t^{p''}A)$. By Lemma~\ref{LemHZ1}, for any given integer $l'$ we may find a sufficiently large $p_1$, such that
\begin{align*}
{}^{y^{-1}x}\bigg[E_n(t^{\LF {\frac{p_1}{2}\RF}}A, t^{\LF {\frac{p_1}{2}\RF}}(IJ+JI)),&{}^c e_{k,j'}(\frac{\gamma}{t^{m+p'}})\bigg]
\in \,  {}^{y^{-1}x}\Big[E_n\big (t^{2l'}A, t^{2l'}(IJ+JI)\big), E_n(t^{2l'}A, t^{2l'}K)\Big]\\
&\subseteq {}^{y^{-1}x}\Big[\big[E_n(t^{l'}A, t^{l'} I),E_n(t^{l'}A, t^{l'}J)\big], E_n(t^{2l'}A, t^{2l'}K)\Big]\\
&\subseteq {}^{y^{-1}x}\Big[\big[E_n(t^{l'}A, t^{l'} I),E_n(t^{l'}A, t^{l'}J)\big], E_n(t^{l'}A, t^{l'}K)\Big]\\
&=\Big[\big[{}^{y^{-1}x}E_n(t^{l'}A, t^{l'} I),{}^{y^{-1}x}E_n(t^{l'}A, t^{l'}J)\big], {}^{y^{-1}x}E_n(t^{l'}A, t^{l'}K)\Big],
\end{align*}
where by definition $y^{-1}x\in E_n(\frac{A}{t^{0}},\frac{A}{t^{0}})$. By (\ref{lem5}) in Lemma~\ref{LemHZ}, for any given integer $l$, we may find a sufficiently large $l'$, such that
$$
{}^{y^{-1}x}\Big[\big[E_n(t^{l'}A, t^{l'} I),E_n(t^{l'}A, t^{l'}J)\big], E_n(t^{l'}A, t^{l'}K)\Big]\subseteq\Big[\big[E_n(t^{l}A, t^{l} I),E_n(t^{l}A, t^{l}J)\big], E_n(t^{l}A, t^{l}K)\Big].
$$
This shows that for any given $l$, one may find a sufficiently large $p_1$ such that the first factor of (\ref{eqn:99})
$$
{}^{y^{-1}x}\bigg[\Big[e_1^{-1}, {}^ce_{i',k}(-t^{p'})\Big],{}^c e_{k,j'}(\frac{\gamma}{t^{m+p'}})\bigg] \in 
\Big[\big[E_n(t^{l}A, t^{l} I),E_n(t^{l}A, t^{l}J)\big], E_n(t^{l}A, t^{l}K)\Big].
$$

Next we consider the second factor of (\ref{eqn:99}),
$$
{}^{y^{-1}z}\bigg[\Big[{}^c e_{k,j'}(\frac{-\gamma}{t^{m+p'}}), e_1\Big],{}^ce_{i',k}(-t^{p'})\bigg].$$
Set $p_2=p$. Note that 
$$e_1\in \big[E_n( t^{p_2} I),E_n(A,  J)\big]\subseteq \GL_n\big(A, t^{p_2}(IJ+JI)\big)$$
and
$$
{}^c e_{k,j'}(\frac{\gamma}{t^{m+p'}})\in  {}^{E_n^1(\frac A {t^{m+p'}})}E_n^1(\frac K {t^{m+p'}}),
$$ where $p'$ is given by~(\ref{eqn:100}) from the first part of the proof. 
We may apply Lemma~\ref{Lem:New1} to find a sufficiently large $p_2$ such that 
\begin{equation}\label{jhafyw}
\Big[{}^c e_{k,j'}(\frac{-\gamma}{t^{m+p'}}), e_1\Big]\in  \GL_n\Big(A, t^{p''} (K(IJ+JI)+(IJ+JI)K)\Big)
\end{equation}
for any given $p''$. Using the commutator formula together with (\ref{eqn:100}), one gets
\begin{eqnarray*}
{}^{y^{-1}z}\bigg[\Big[{}^c e_{k,j'}(\frac{-\gamma}{t^{m+p'}}), e_1\Big],{}^ce_{i',k}(-t^{p'})\bigg]&\in&
{}^{y^{-1}z} E_n\Big(A, t^{p''} \big(K(IJ+JI)+(IJ+JI)K\big)\Big)
\end{eqnarray*}
Applying  Lemma~\ref{Lem:Habdank} twice, one gets
\begin{eqnarray*}
E_n\Big(A, t^{p''} \big(K(IJ+JI)+(IJ+JI)K\big)\Big)&\subseteq& \bigg[E_n\Big(t^{\LF\frac {2p''}{3}\RF} \big((IJ+JI)+(IJ+JI)\big)\Big),E_n\Big(t^{\LF\frac {p''}{3}\RF}K\Big)\bigg]\\
&\subseteq&\bigg[\Big[E_n(t^{\LF\frac {p''}{3}\RF}I),E_n(t^{\LF\frac {p''}{3}\RF}J)\Big],E_n(t^{\LF\frac {p''}{3}\RF}K)\bigg].
\end{eqnarray*}
Hence, we have
\begin{eqnarray*}
{}^{y^{-1}z}\bigg[\Big[{}^c e_{k,j'}(\frac{-\gamma}{t^{m+p'}}), e_1\Big],{}^ce_{i',k}(-t^{p'})\bigg]&\subseteq&
{}^{y^{-1}z}\bigg[\Big[E_n(t^{\LF\frac {p''}{3}\RF}I),E_n(t^{\LF\frac {p''}{3}\RF}J)\Big],E_n(t^{\LF\frac {p''}{3}\RF}K)\bigg]\\
&=&\bigg[\Big[{}^{y^{-1}z}E_n(t^{\LF\frac {p''}{3}\RF}I),{}^{y^{-1}z}E_n(t^{\LF\frac {p''}{3}\RF}J)\Big],{}^{y^{-1}z}E_n(t^{\LF\frac {p''}{3}\RF}K)\bigg].
\end{eqnarray*}
Now applying (\ref{lem5}) in Lemma~\ref{LemHZ} to every components of the commutator above, we may find a sufficiently large $p''$ such that for any given $l$,
$$
\bigg[\Big[{}^{y^{-1}z}E_n(t^{\LF\frac {p''}{3}\RF}I),{}^{y^{-1}z}E_n(t^{\LF\frac {p''}{3}\RF}J)\Big],{}^{y^{-1}z}E_n(t^{\LF\frac {p''}{3}\RF}K)\bigg]\subseteq \Big[\big[E_n(t^{l}A, t^{l} I),E_n(t^{l}A, t^{l}J)\big], E_n(t^{l}A, t^{l}K)\Big].
$$
Choose $p_2$ in (\ref{jhafyw}) according to this $p''$ and then consider $p$ to be the largest among $p_1$ and $p_2$. This finishes the Lemma.
\end{proof}

\section{Main result}
Now we are in a position to prove the main result of the paper, namely the multiple commutator formulas. We first study  the $3$-folded commutator formula. The $n$-folded commutator formula is an easy application of the following theorem.
 Note that so far most of the conjugation calculus has been performed in $A_t$. Using the fact that for a suitable positive integer $l$, by Lemma~\ref{Lem:03}, the restriction of $\theta_t$ to $\GL_n(A,t^lA)$ induces an injective homomorphism $\theta_t: \GL_n(A,t^lA) \rightarrow \GL_n(A_t)$, we are able to ``pull back'' the elements into the group $\GL_n(A)$. This will be used in the Theorem~\ref{main}.

\begin{The}\label{main}
Let $A$ be a quasi-finite $R$-algebra and $I, J, K$ be two-sided ideals of $A$. Then for $n\ge 3$, 
\begin{equation}\Label{eq:final}
\Big[\big[E_n(A,I),\GL_n(A,J)\big],\GL_n(A,K)\Big]=
\Big[\big[E_n(A,I),E_n(A,J)\big],E_n(A,K)\Big]
\end{equation}
\end{The}
\begin{proof}
The functors $E_n$ and $\GL_n$ commute with direct limits. By Proposition~\ref{Prop:01} and \S\ref{sub:1.3}, one reduces the proof  to the case $A$ is finite over $R$ and $R$ is Noetherian. 

First by the generalized commutator formula, we have
\begin{equation}\label{fin:1}
\big[E_n(A,I),\GL_n(A,J)\big]=\big[E_n(A,I),E_n(A,J)\big].
\end{equation}
Thus it suffices  to prove the following equation
$$
\Big[\big[E_n(A,I),E_n(A,J)\big],\GL_n(A,K)\Big]=
\Big[\big[E_n(A,I),E_n(A,J)\big],E_n(A,K)\Big].
$$
By Lemma~\ref{Comgenerator}, $\big[E_n(A,I),E_n(A,J)\big]$ is generated by the conjugates of the following four  types of elements
$$e=\Big[e_{j,i}(\alpha), {}^{e_{i,j}(r)}e_{j,i}(\beta)\Big],\quad e=\big[e_{j,i}(\alpha), e_{i,j}(\beta)\big],\quad e=e_{i,j}(\alpha\beta),\quad \text{and    } e=e_{i,j}(\beta\alpha),$$
where $i\ne j$, $\alpha\in I$, $\beta\in J$. We claim that for any $g\in E_n(A,K)$, $$\big[e,g] \in \Big[\big[E_n(A,I),E_n(A,J)\big],E_n(A,K)\Big].$$

For any maximal ideal $\gm_l\lhd R$, choose a $ t_l \in R\backslash \gm_l$ and an arbitrary positive integer $p_l$. (We will later choose $p_k$ according to Lemma~\ref{Lem:08}.)  
Since the collection of all $t_l^{p_l}$ is not contained in any maximal ideal, we may find a finite number of $t_l$ and $x_l\in R$, $l=1,\dots,k$ (relabeling if necessary) such that
$$
\sum_{l}t_l^{p_l}x_l=1.
$$
First we take the generators of the first kind, namely the conjugates of $e=\Big[e_{j,i}(\alpha),{}^{e_{i,j}(r)}e_{j,i}(\beta)\Big]$. Consider
$$
e=\Big[e_{j,i}(\alpha),{}^{e_{i,j}(r)}e_{j,i}(\beta)\Big]=\Big[e_{j,i}\Big((\sum_{l}t_l^{p_l}x_l)\alpha\Big),{}^{e_{i,j}(r)}e_{j,i}(\beta)\Big]=\Big[\prod_{l}e_{j,i}(t_l^{p_l}x_l\alpha),{}^{e_{i,j}(r)}e_{j,i}(\beta)\Big].
$$
By (C$2^+$) identity, $e=\displaystyle \Big[\prod_{l}e_{j,i}(t_l^{p_l}x_l\alpha),{}^{e_{i,j}(r)}e_{j,i}(\beta)\Big]$ can be written as a product of the following form:
\begin{equation}\label{hfsis}
e=\Big({}^{e_1}\Big[e_{j,i}(t_1^{p_1}x_1\alpha),{}^{e_{i,j}(r)}e_{j,i}(\beta)\Big]\Big)\Big({}^{e_2}\Big[e_{j,i}(t_2^{p_2}x_2\alpha),{}^{e_{i,j}(r)}e_{j,i}(\beta)\Big]\Big)\cdots \Big({}^{e_k}\Big[e_{j,i}(t_k^{p_k}x_k\alpha),{}^{e_{i,j}(r)}e_{j,i}(\beta)\Big]\Big),
\end{equation}
where $e_1,e_2,\ldots e_m\in E_n(A)$. Note that all $e_i$'s are products of elementary matrices of the form $e_{j,i}(A)$. Thus $e_l=e_{j,i}(a_l)$, $l=1,\dots,k$, which clearly commutes with $e_{j,i}(a)$ for any $a\in A$. So the commutator~(\ref{hfsis}) equals to 
\begin{equation}\label{hfsis1}
e=\Big(\Big[e_{j,i}(t_1^{p_1}x_1\alpha),{}^{e_1}{}^{e_{i,j}(r)}e_{j,i}(\beta)\Big]\Big)\Big(\Big[e_{j,i}(t_2^{p_2}x_2\alpha),{}^{e_2}{}^{e_{i,j}(r)}e_{j,i}(\beta)\Big]\Big)\cdots \Big(\Big[e_{j,i}(t_k^{p_k}x_k\alpha),{}^{e_k}{}^{e_{i,j}(r)}e_{j,i}(\beta)\Big]\Big).
\end{equation}

Thus $$[e,g]=\bigg[\Big[e_{j,i}(\alpha),{}^{e_{i,j}(r)}e_{j,i}(\beta)\Big],g\bigg]=\bigg[\displaystyle \Big[\prod_{l}e_{j,i}(t_l^{p_l}x_l\alpha),{}^{e_{i,j}(r)}e_{j,i}(\beta)\Big],g\bigg]. 
$$ Using (C$2^+)$ and in view of~(\ref{hfsis1}) we obtain that $[e,g]$ is a  product of the conjugates in $E_n(A)$ of $$\bigg[\Big[e_{j,i}(t_{i'}^{p_{i'}}x_{i'}\alpha),{}^{e_{j,i}(a_l)e_{i,j}(r)}e_{j,i}(\beta)\Big],g\bigg],$$ where $a_l \in A$ and $l=1,\ldots,k$.
 
For any maximal ideal $\gm$ of $R$, the ring $A_\gm$ contains $K_\gm$ as an ideal.   
Consider the natural homomorphism $\theta_\gm:A\rightarrow A_\gm$ which induces a homomorphism (call it again $\theta_\gm$) on the level of general linear groups, $\theta_\gm:\GL_{n}(A)\rightarrow \GL_{n}(A_\gm)$. 
 Therefore, for  $g\in \GL_n(A, K)$, $\theta_\gm(g)\in \GL_n(A_\gm,K_\gm)$. 
 Since $A_\gm$ is module finite over the local ring $R_\gm$, $A_\gm$ is semilocal  \cite[III(2.5), (2.11)]{Bass1}, therefore  its stable rank is $1$. It follows that  $\GL_n(A_\gm,K_\gm)=E_n(A_\gm,K_\gm)\GL_1(A_\gm,K_\gm)$ (see \cite[Th.~4.2.5]{HO}).  So $\theta_\gm(g)$ can be decomposed as $\theta_\gm(g)=\ep h$, where $\ep \in E_n(A_m, K_m)$ and $h$ is a diagonal matrix all of whose diagonal coefficients are $1$, except possibly the $k$-th diagonal coefficient, and $k$ can be chosen arbitrarily.

By (\S\ref{sub:1.3}), we may reduce the problem  to the case $A_t$ with $t\in R\backslash \gm$. Namely $\theta_t(g)$ is a product of $\ep$ and $h$, where $\ep\in E_n(A_t, K_t)$, and $h$ is a diagonal matrix with only one non-trivial diagonal entry which lies in $A_t$.   Note that all $\Big[e_{j,i}(t_{i'}^{p_{i'}}x_{i'}\alpha),{}^{e_{j,i}(a_l)e_{i,j}(r)}e_{j,i}(\beta)\Big]$ with $i'=1,\ldots, k$  differ from the identity matrix at only the $i,j$ rows and the $i,j$ columns. By the assumption $n>2$, we may choose $h$ so that it commutes with 
$\Big[e_{j,i}(t_{i'}^{p_{i'}}x_{i'}\alpha),{}^{e_{j,i}(a_l)e_{i,j}(r)}e_{j,i}(\beta)\Big]$. 
This allows us to reduce our consideration to the case 
$$
\bigg[\Big[e_{j,i}(t_{i'}^{p_{i'}}x_k\alpha),{}^{e_{j,i}(a_l)e_{i,j}(r)}e_{j,i}(\beta)\Big], \ep\bigg].
$$
By Lemma~\ref{Lem:08}, one gets that for any given $l_{i'}$, there is a sufficiently large $p_{i'}$ for every  $i'=1,\dots k$, such that
$$
\bigg[\Big[e_{j,i}(t_{i'}^{p_{i'}}x_k\alpha),{}^{e_{j,i}(a_l)e_{i,j}(r)}e_{j,i}(\beta)\Big], \ep\bigg]
\in\Big[\big[E_n(A,t^{l_{i'}}I), E_n(A,t^{l_{i'}}J)\big], E_n(A,t^{l_{i'}}K)\Big].
$$
Let's choose every $l_i$ to be large enough so that the restriction of $\theta_t:  GL_n(t^{l_{i'}}A)\to GL_n(A_t)$ is injective. Then for every $i'$, we have
$$\bigg[\Big[e_{j,i}(t_{i'}^{p_{i'}}x_k\alpha),{}^{e_{j,i}(a_l)e_{i,j}(r)}e_{j,i}(\beta)\Big], \ep\bigg]\in 
\Big[\big[E_n(A,I),E_n(A,J)\big],E_n(A,K)\Big].$$
Hence $[e,g]\in \Big[\big[E_n(A,I),E_n(A,J)\big],E_n(A,K)\Big].$

When the generator is of the second kind, $e=[e_{i,j}(\alpha),e_{j,i}(\beta)]$, a similar argument goes through, which is left to the reader.

Now consider the generators of the 3rd and 4th kind, namely, the conjugates of the following two types of elements,
$$e=e_{i,j}(\alpha\beta),\quad \text{or    } e=e_{i,j}(\beta\alpha).$$
By the normality of $E_{n}(A,IJ+JI)$,  the conjugates of $e$ are in $ E_{n}(A, IJ+JI)$. Then
$$
[e,g]\in \big[E_{n}(A, IJ+JI), \GL_n(A,K)\big].
$$
By the generalized commutator formula, one obtains 
$$
\big [E_{n}(A, IJ+JI), \GL_n(A,K)\big]=\big[E_{n}(A, IJ+JI), E_n(A,K)\big].
$$
Now applying Lemma~\ref{Lem:Habdank}, we finally get 
$$
[E_{n}(A, IJ+JI), E_n(A,K)]\subseteq \Big[\big[E_n(A,I),E_n(A,J)\big],E_n(A,K)\Big].
$$
Therefore  $[e,g]\in \Big[\big[E_n(A,I),E_n(A,J)\big],E_n(A,K)\Big].$
This proves our claim.

Let $e \in \big[E_n(A,I),\GL_n(A,J)\big]$, and $g\in \GL_n(A,K)$. Then by Lemma~\ref{Comgenerator},
$$e={}^{c_1}e_{1}\times{}^{c_2}e_{2}\times\cdots\times {}^{c_k}e_{k}$$
with $c_i \in E_n(A)$ and $e_{i}$ takes the form in (\ref{yyttrree}). Thanks to (C$2^+$) identity, it suffices to show that
$$[{}^{c_1}e_{1},g]\in\Big[\big[E_n(A,I),E_n(A,J)\big],E_n(A,K)\Big].$$
The normality of $E_n$ and $\GL_n$ groups reduces the problem to show that
$$[e_i,g]\in\Big[\big[E_n(A,I),E_n(A,J)\big],E_n(A,K)\Big] ,\quad i=1,\dots,k.$$
But this exactly what has been shown above. This completes the proof. 
\end{proof}


\begin{Cor}\label{comain}
Let $A$ be  a quasi-finite ring with identity and  $I_i $, $i=0,...,m$, be two-sided ideals of $A$. Then 
\begin{equation}\label{corll11}
\begin{split}
\Big [E_n(A,I_0),\GL_n(A,I_1),& \GL_n(A, I_2),\ldots, \GL_n(A, I_m) \Big]\\
&=\Big[E_n(A,I_0),E_n(A,I_1),E_n(A, I_2),\ldots, E_n(A, I_m) \Big].
\end{split}
\end{equation}
\end{Cor}
\begin{proof}
We prove the statement by induction. For $i=1$ this is the generalized commutator formula $$\big[E_n(A,I_0),\GL_n(A,I_1)\big]=\big[E_n(A,I_0),E_n(A,I_1)\big]$$ which was proved in \cite{RHZZ1}. For $i=2$, this is proved in Theorem~\ref{main} which will be the first step of induction.
Suppose the statement is valid for $i=m-1$ (i.e., there are $m$ ideals in the commutator formula). To prove ~(\ref{corll11}), using Theorem~\ref{main}, we have
\[
\begin{split}
 \bigg [\Big[\big[E_n(A,I_0),\GL_n(A,I_1)\big], \GL_n(A, I_2)\Big],&\GL_n(A, I_3),\ldots, \GL_n(A, I_m) \bigg ]=\\
& \bigg [\Big[\big[E_n(A,I_0),E_n(A,I_1)\big ],E_n(A, I_2)\Big],\GL_n(A, I_3),\ldots, \GL_n(A, I_m) \bigg].
\end{split}
\]
 By Lemma~\ref{Lem:Habdank}, $[E_n(A,I_0),E_n(A,I_1)]\subseteq \GL_n(A,I_0I_1+I_1I_0)$. Thus 
\[
\begin{split}
 \bigg [\Big[\big[E_n(A,I_0),E_n(A,I_1)\big ],E_n(A, I_2)\Big],&\GL_n(A, I_3),\ldots, \GL_n(A, I_m) \bigg] \subseteq \\
 & \bigg [\Big[\GL_n(A,I_0I_1+I_1I_0),E_n(A, I_2)\Big],\GL_n(A, I_3),\ldots, \GL_n(A, I_m) \bigg].
\end{split}
\]
Since there are $m$ groups involved in the commutator subgroups in the right hand side, by induction we get 
\[
\begin{split}
 \bigg [\Big[\GL_n(A,I_0I_1+I_1I_0),E_n(A, I_2)\Big] &,\GL_n(A, I_3),\ldots, \GL_n(A, I_m) \bigg]=\\
& \bigg [\Big[E_n(A,I_0I_1+I_1I_0),E_n(A, I_2)\Big],E_n(A, I_3),\ldots, E_n(A, I_m) \bigg].
\end{split}
\]
Finally again by Lemma~\ref{Lem:Habdank}, $E_n(A,I_0I_1+I_1I_0) \subseteq \big[E_n(A,I_0),E_n(A,I_1)\big]$. Replacing this in the above equation we obtain that the left hand side of~(\ref{corll11}) is contained in the right hand side. The opposite inclusion is obvious. This completed the proof. 
\end{proof}

The following corollary shows in fact it doesn't matter where the elementary subgroup appears in the multiple commutator formula. 

\begin{Cor}\label{comain2}
Let $A$ be  a quasi-finite ring with identity and $I_i $, $i=0,...,m$, be two-sided ideals of $A$. 
Let $G_i$ be  subgroups of $\GL_n(A)$ such that 
$$
E_n(A,I_i) \subseteq G_i\subseteq \GL_n(A, I_i),\quad\text{ for } i=0,\ldots, m.
$$
If there is an index $j$ such that $G_j=E_n(A,I_j)$, then 
$$
\big[G_0,G_1,\ldots, G_m\big]=\Big[E_n(A,I_0),E_n(A,I_1),E_n(A, I_2),\ldots, E_n(A, I_m) \Big].
$$

\end{Cor}
\begin{proof}
Define two-sided ideals of $A$ inductively as follows
\begin{eqnarray*}
\mathcal I_0&=&I_0 \\ \mathcal I_{k}&=&I_{k}\mathcal I_{k-1} +\mathcal I_{k-1} I_{k},
\end{eqnarray*} where $k=1,\dots,m$. 
Now the proof of the lemma divides into several cases:

If $j=0$, the proof follows directly from Corollary~\ref{comain}. If $j=1$, by a basic property of $2$-fold commutator subgroups,  
$\big[G_0, E_n(A,I_1)\big]=\big[E_n(A,I_1), G_0\big]$, so we reduce the problem to the case of $j=0$.

When $j=k$ with $k\ge 2$, we have 
\begin{eqnarray*}
\big [G_0,G_1,\ldots, G_k, G_{k+1},\ldots, G_m\big]&=&\Big [\big[G_0,G_1,\ldots G_k\big], G_{k+1},\ldots, G_m\Big]\\
&=&\Big[\big[G_0,G_1,\ldots, G_{k-1}, E_n(A,I_k)\big], G_{k+1},\ldots, G_m\Big].
\end{eqnarray*}
Furthermore,  $\big [G_0,G_1,\ldots, E_n(A,I_k)\big]=\Big[\big[G_0,G_1,\ldots, G_{k-1}\big], E_n(A,I_k)\Big]$,  and it follows from Lemma~\ref{Lem:New1}, by putting $t=1$,  that $\big[G_0,G_1,\ldots, G_{k-1}\big]\subseteq \GL_n(A, \mathcal I_{k-1})$. By the generalized commutator formula and Lemma~\ref{Lem:Habdank} we have 
\begin{align*}
\Big [\big [E_n(A,I_0),E_n(A,I_1),\ldots, E_n(A,I_{k-1})\big], E_n(A,I_k)\Big] & \subseteq \Big [\big [G_0,G_1,\ldots G_{k-1}\big], E_n(A,I_k)\Big]\\
& \subseteq  \big [\GL_{n}(A, \mathcal I_{k-1}), E_n(A,I_k)\big]\\
& =  \big [E_{n}(A, \mathcal I_{k-1}), E_n(A,I_k)\big]\\
&\subseteq \Big [\big [E_n(A,I_0),E_n(A,I_1),\ldots, E_n(A,I_{k-1})\big], E_n(A,I_k)\Big].
\end{align*}
So 
$$
\big [G_0,G_1,\ldots, G_{k-1},E_n(A,I_k)\big]=\Big [\big [E_n(A,I_0),E_n(A,I_1),\ldots, E_n(A,I_{k-1})\big], E_n(A,I_k)\Big],
$$
and therefore  
$$
\big [G_0,G_1,\ldots, G_m\big ]=\big [E_n(A,I_0),E_n(A,I_1),\ldots, E_n(A,I_{k-1}), E_n(A,I_k),G_{k+1},\ldots, G_{m}\big].
$$
Finally, we finish the proof by applying Corollary~\ref{comain}. 
\end{proof}
We finish the paper by a most general multiple commutator formula.
Note that taking a commutator is a binary operation, and for $G_1,\dots,G_n$, $n\geq 3$, there any many ways to insert the commutator  brackets $[\_, \_ ]$ to make the sequence into a meaningful multi-commutator expression. For example for $n=4$, we can have the following two arrangements  $\Big[\big [G_0,[G_1,G_2]\big], G_3\Big]$ and $\Big[\big[G_0,G_1\big],\big[G_2, G_3\big]\Big]$ among many others. We denote by $\big \llbracket G_0,G_1,\ldots,G_m \big \rrbracket$ a ``meaningful'' multi-commutator formula.

\begin{The}\Label{main2}
Let $A$ be  a quasi-finite ring with identity and $I_i $, $i=0,...,m$, be two-sided ideals of $A$. Let $G_i$ be  subgroups of $\GL_n(A)$ such that 
$$
E_n(A,I_i)\subseteq G_i\subseteq \GL_n(A, I_i),\quad\text{ for } i=0,\ldots, m.
$$
If there is an index $j$ such that $G_j=E_n(A,I_j)$, then 
$$
\big \llbracket G_0,G_1,\ldots, G_m\big \rrbracket=\big \llbracket E_n(A,I_0),E_n(A,I_1),\ldots,E_n(A,I_m) \big \rrbracket.
$$
\end{The}
\begin{proof}
For simplicity denote $E_n(A,I_i)$ by $E_i$. 
The proof is by induction on $m$. For $m=0$ and $m=1$ there is nothing to prove.   For m=2, the commutator  $\big \llbracket G_0,G_1, G_2\big \rrbracket$ can take one of the the forms 
\begin{enumerate}
\item $\Big[\big[G_0,G_1\big], E_2\Big ],$ 
\item $\Big[E_0,\big[G_1,G_2\big]\Big ],$
\item $\Big[\big[E_0,G_1\big], G_2\Big ],$ 
\item $\Big[\big[G_0,E_1\big], G_2\Big ],$ 
\item  $\Big[G_0,\big[E_1,G_2\big]\Big ],$
\item $\Big[G_0,\big[G_1,E_2\big]\Big ]$.
\end{enumerate}
 
Since by (C(6)), for two subgroups $H$ and $K$, we have $[H,K]=[K,H]$, we can reduce the cases (1) and (2) and (3)--(6) to each other, respectively, and therefore it is enough to prove the theorem for the cases (1) and (3). For the first arrangement (1),
using Lemma~\ref{Lem:New1}, for $t=1$, Lemma~\ref{Lem:Habdank} and the generalized commutator formula we have 
\begin{align*}
\big \llbracket E_0,E_1, E_2\big \rrbracket=\Big[\big[E_0,E_1\big], E_2\Big ] &  \subseteq \Big[\big[G_0,G_1\big], E_2\Big ]\\
& \subseteq \Big[\big[\GL_n(A,I_0),\GL_n(A,I_1)\big], E_n(A,I_2)\Big ]\\
& \subseteq \Big[\GL_n(A,I_0I_1+I_1I_0), E_n(A,I_2)\Big ]\\
& = \Big[E_n(A,I_0I_1+I_1I_0), E_n(A,I_2)\Big ]\\
& \subseteq \Big[\big[E_n(A,I_0),E_n(A,I_1)\big], E_n(A,I_2)\Big ]\\
&= \Big[\big[E_0,E_1\big], E_2\Big ]  = \big \llbracket E_0,E_1, E_2\big\rrbracket. 
\end{align*}
This shows that $\big \llbracket G_0,G_1, G_2\big\rrbracket=\big \llbracket E_0,E_1, E_2\big\rrbracket.$ 
The arrangement (3) (and therefore (4)--(6)) follows immediately from Theorem~\ref{main}. 

For the main step of induction, we consider two cases. Suppose first there is  a mixed commutator $[G_i,G_{i+1}]$ in $\big \llbracket G_0,G_1,\ldots, G_m\big \rrbracket$, where neither $G_i$ nor $G_{i+1}$ is the {\it fixed} elementary subgroup $E_j$. Then 
\begin{align}
\big \llbracket G_0,G_1,\ldots,G_m \big \rrbracket & =\big \llbracket G_0,G_1,\ldots,[G_i,G_{i+1}],\dots , G_m \big \rrbracket\notag  \\
&\subseteq \big \llbracket G_0,G_1,\ldots,\big [\GL_n(A,I_i),\GL_n(A,I_{i+1})\big],\dots , G_m \big \rrbracket \label{ryan1} \\
&\subseteq \big \llbracket G_0,G_1,\ldots,\GL_n(A,I_iI_{i+1}+I_{i+1}I_i),\dots , G_m \big \rrbracket.\notag
\end{align}
Note that there are one fewer ideal involved in the last commutator formula (i.e., $m-1$ ideals) which also contains an elementary subgroup, and so by induction 
\begin{align}
\big \llbracket G_0,G_1,\ldots,\GL_n(A,I_iI_{i+1}+I_{i+1}I_i),\dots, G_m \big \rrbracket &=\big \llbracket E_0,E_1,\ldots,E_n(A,I_iI_{i+1}+I_{i+1}I_i),\dots, E_m\big \rrbracket \notag \\
&\subseteq \big \llbracket E_0,E_1,\ldots,\big[E_n(A,I_i),E_n(A,I_{i+1})\big],\dots, E_m\big \rrbracket \label{ryan2} \\
&= \big \llbracket E_0,E_1,\ldots,E_m \big \rrbracket. \notag
\end{align}
Putting~\ref{ryan1} and~\ref{ryan2} together, we get 
$$
\big \llbracket G_0,G_1,\ldots,G_m \big \rrbracket = \big \llbracket E_0,E_1,\ldots,E_m \big \rrbracket.
$$
For the remaining case, suppose now that if there is a mixed commutator of the form $[G_i,G_{i+1}]$, in $\big \llbracket G_0,G_1,\ldots, G_m\big \rrbracket$, then one of $G_i$ or $G_{i+1}$ is our fixed elementary subgroup $E_j$. Write 
$$
\big \llbracket G_0,G_1,\ldots, G_m\big \rrbracket=\Big[\big \llbracket G_0,G_1,\ldots, G_k\big \rrbracket, \big \llbracket G_{k+1},\ldots, G_m\big \rrbracket \Big].
$$ Since the fixed elementary subgroup $E_j$ is in one of the factors, one of  $\big \llbracket G_0,G_1,\ldots, G_k\big \rrbracket$ or $ \big \llbracket G_{k+1},\ldots, G_m\big \rrbracket$ has to have a mixed commutator of the form $[G_{i'},G_{i'+1}]$ with neither $G_{i'}$ nor $G_{i'+1}$ the fixed $E_j$, which has been excluded from the outset. This forces $k=0$ or $k=m-1$, i.e., 
\begin{align*}
\big \llbracket G_0,G_1,\ldots, G_m\big \rrbracket & =\Big[G_0, \big \llbracket G_{1},\ldots, G_m\big \rrbracket \Big], \text{ or     }  \\
\big \llbracket G_0,G_1,\ldots, G_m\big \rrbracket & =\Big[\big \llbracket G_0,G_1,\ldots, G_{m-1}\big \rrbracket, G_m \Big].
\end{align*}
Repeating this argument, by an easy induction and (C(6)) one can see that because of absence of $[G_i,G_{i+1}]$, the multiple commutator $\big \llbracket G_0,G_1,\ldots, G_m\big \rrbracket$ has the standard form (see \S\ref{sub:1.4})
$$
\big \llbracket G_0,G_1,\ldots, G_m\big \rrbracket =\big [G_{i_0},G_{i_1},\dots,G_{i_m}\big].
$$ By Corollary~\ref{comain2} 
$$
\big[G_{i_0},G_{i_1},\dots,G_{i_m}\big]=\big[E_{i_0},E_{i_1},\dots,E_{i_m}\big]. 
$$
Now using (C(6)) again and re-arranging $E_i$ in the reverse order we get
$$
\big[E_{i_0},E_{i_1},\dots,E_{i_m}\big]=\big \llbracket E_0,E_1,\ldots, E_m\big \rrbracket. 
$$ This finishes the proof.
\end{proof}

\end{document}